\documentclass[11pt]{amsart}
\usepackage{amssymb}

\theoremstyle{plain}
\newtheorem{theorem}{Theorem}
\newtheorem{corollary}{Corollary}

\newtheorem{lemma}{Lemma}

\theoremstyle{definition}

\theoremstyle{remark}

\numberwithin{equation}{section}

\setbox0=\hbox{$+$}

\newdimen\plusheight
\plusheight=\ht0
\def\+{\;\lower\plusheight\hbox{$+$}\;}

\setbox0=\hbox{$-$}
\newdimen\minusheight
\minusheight=\ht0
\def\-{\;\lower\minusheight\hbox{$-$}\;}

\setbox0=\hbox{$\cdots$}
\newdimen\cdotsheight
\cdotsheight=\plusheight
\def\cds{\lower\cdotsheight\hbox{$\cdots$}}

\begin{document}

%
%


\title[$m$-versions of Some Identities
of Rogers-Ramanujan-Slater Type] {Continued Fraction Proofs  of  $m$-versions
of Some Identities of Rogers-Ramanujan-Slater Type }

 \author{Douglas Bowman}
\address{ Northern Illinois University\\
   Mathematical Sciences\\
   DeKalb, IL 60115-2888 }
 \email{bowman@math.niu.edu}

 \author{James Mc Laughlin}
\address{Mathematics Department\\
 25 University Avenue\\
West Chester University, West Chester, PA 19383}
\email{jmclaughl@wcupa.edu}

\author{ Nancy J. Wyshinski}
\address{Mathematics Department\\
       Trinity College\\
        300 Summit Street, Hartford, CT 06106-3100}
\email{nancy.wyshinski@trincoll.edu}

\begin{abstract}
We derive two general transformations for certain basic
hypergeometric series from the recurrence formulae for the partial
numerators and denominators of two $q$-continued fractions
previously investigated by the authors.

By then specializing certain free parameters in these
transformations, and employing various identities of
Rogers-Ramanujan type, we derive \emph{$m$-versions} of these
identities. Some of the identities thus
found are new, and some have been derived previously by other
authors, using different methods.

By applying certain transformations due to Watson, Heine and
Ramanujan, we derive still more examples of such $m$-versions of Rogers-Ramanujan-type identities.
\end{abstract}



\maketitle

\section{Introduction}
In \cite{GIS99}, the authors prove the following generalization of
the well-known Rogers-Ramanujan identities. For an integer $m\geq 0$,
\begin{equation}\label{gisrr}
\sum_{n=0}^{\infty}\frac{q^{n^2+m n}}{(q;q)_n}=\frac{(-1)^m q^{-m(m-1)/2}a_m(q)}{(q,q^4;q^5)}+\frac{(-1)^{m+1} q^{-m(m-1)/2}b_m(q)}{(q^2,q^3;q^5)},
\end{equation}
where $a_0(q)=1$, $b_0(q)=0$, and for $m\geq 1$,
\begin{align*}
a_m(q)&=\sum_{n} q^{n^2+n}\left [
\begin{matrix}
m-2-n\\
n
\end{matrix}
\right ],\\
b_m(q)&=\sum_{n} q^{n^2}\left [
\begin{matrix}
m-1-n\\
n
\end{matrix}
\right ].
\end{align*}
The cases $m=0$ and $m=1$ give the original Rogers-Ramanujan
identities, and following the authors in \cite{IS03}, we refer to
\eqref{gisrr} as an ``$m$-version" of the Rogers-Ramanujan
identities.

We will use the term ``$m$-version" in the paper to mean an identity involving an
 integer parameter $m$, such that setting $m=0$ or $m=1$ will
 recover a known $q$-series-product identity, such as may be found on the ``Slater list" (see \cite{S52} or \cite{S03}) or elsewhere.

The authors in \cite{GIS99} derived \eqref{gisrr} by evaluating  an
integral involving $q$-Hermite polynomials in two different ways and
equating the results. In \cite{AKP00}, the authors derive
\eqref{gisrr} using their method of Engel Expansions, and indeed
give a polynomial generalization of \eqref{gisrr}, similar in nature
to those found at \eqref{polyver} and \eqref{polyver2} below.
Garrett \cite{G05} stated a more general identity that contains the
polynomial identity of Andrews et al. in \cite{AKP00} as a special
case, and gave an $m$-version of another pair of identities on the
Slater list \cite{S52} as another application. Other applications of
Engel expansions to produce similar identities are given in
\cite{AKPP01}.

In \cite{IPS00}, determinant methods are use to derive a
generalization of \eqref{gisrr}, as well as an $m$-version of
Identities \textbf{38} and \textbf{39} in \cite{S52} by Slater (for
a different $m$-version of these identities, see Corollary \ref{cc2}
 below). In \cite{GP09}, determinant methods are
further explored to derive
 $m$-versions of several other Slater identities, making use of
 the finitizations of these Slater identities given in
 \cite{S03} by Sills. The methods employed in \cite{GP09} also allow the
 authors to derive $m$ versions of certain \emph{triples} of
 identities given in \cite{S52}.

 In \cite{IS03}, Ismail and Stanton gave a formula for the right
 side of \eqref{gisrr} in the case where $m$ is a negative integer,
 and, amongst other results, gave several other  similar $m$-versions of
 Slater identities in the case where $m$ is a negative integer.

In the present paper we show that $m$-versions of several pairs of
identities of Rogers-Ramanujan type may be derived easily from
known properties of two quite general $q$-continued fractions
previously investigated by the present authors. Continued fraction
methods give an easy approach to proving this type of identity
in many cases, as once the general result from the $q$-continued fraction is
made explicit, all that is necessary to produce $m$-versions of
certain identities of Rogers-Ramanujan type is to specialize various
parameters.

We derive still further $m$-version identities by applying
transformations due to Watson, Heine and Ramanujan to the
$m$-version identities we derived using continued fraction methods.

We also derive negative $m$-versions of several identities, by similar methods.

\section{A General Basic Hypergeometric Transformation}

In this section we prove a general transformation (Theorem
\ref{t1ef} below)  for certain basic hypergeometric series, a
transformation which gives \eqref{gisrr} and several similar
formulae as special cases.

We first recall some properties of continued fractions which will be
used later. Let $P_{n}$ denote the $n$-th numerator convergent,
and $Q_{n}$ denote the $n$-th  denominator convergent, of the
continued fraction
\begin{equation*}
b_{0} + \frac{a_{1}}{b_{1}}
\+
 \frac{a_{2}}{b_{2}}
\+
 \frac{a_{3}}{b_{3}}
\+\,\cds
\end{equation*}
Then (see, for example, \cite{LW92}, p.9)
 the $P_{n}$'s and $Q_{n}$'s
satisfy the following recurrence relations.
\begin{align}\label{recurrel}
P_{n}&=b_{n}P_{n-1}+a_{n}P_{n-2},\\
Q_{n}&=b_{n}Q_{n-1}+a_{n}Q_{n-2}.\notag
\end{align}
It is also well known (see also \cite{LW92}, p.9) that, for $n \geq 1$,
\begin{align}\label{reclem}
P_{n}Q_{n-1}-P_{n-1}Q_{n-1} &= (-1)^{n-1}\prod_{i=1}^{n}a_{n}.
\end{align}

We also recall the $q$-binomial theorem (\cite{A76}, pp. 35--36).
\begin{lemma}\label{qbin}
If $\left [
\begin{matrix}
n\\
m
\end{matrix}
\right ] $ denotes the Gaussian polynomial defined by
\[
\left [
\begin{matrix}
n\\
m
\end{matrix}
\right ] := \left [
\begin{matrix}
n\\
m
\end{matrix}
\right ]_{q} :=
\begin{cases}
\displaystyle{
\frac{(q;q)_{n}}{(q;q)_{m}(q;q)_{n-m}}}, &\text{ if } 0 \leq m \leq n,\\
0, &\text{ otherwise },
\end{cases}
\]
then
\begin{align}\label{qbineq}
&(z;q)_{N}= \sum_{j=0}^{N}\left [
\begin{matrix}
N\\
j
\end{matrix}
\right ]
(-1)^{j}z^{j}q^{j(j-1)/2},\\
&\frac{1}{(z;q)_{N}} =\sum_{j=0}^{\infty} \left [
\begin{matrix}
N+j-1\\
j
\end{matrix}
\right ] z^{j}. \notag
\end{align}
\end{lemma}

In \cite{BMcLW06} the following result (slightly rephrased) was proven.
\begin{theorem}\label{qth}
Let \[ H(a,b,c,d,q)= \frac{1}{1} \+ \frac{-ab+cq}{a+b+dq}
\+\frac{-ab+cq^2}{a+b+dq^2} \+ \cds \+ \frac{-ab+cq^n}{a+b+dq^n}\+
\cds .
\]
Let $A_{N}:=A_{N}(q)$ and $B_{N}:=B_{N}(q)$ denote the $N$-th
numerator convergent and $N$-th denominator convergent,
respectively, of $H(a,b,c,d,q)$. Then $A_{N}$ and $B_{N}$ are
given explicitly by the following formulae.
\begin{multline}\label{ANeq}
A_{N}=b^{N-1}\sum_{n\geq 0} (d/b)^{n}q^{n(n+1)/2}
 \sum_{j\geq 0}
\left [
\begin{matrix}
n+j\\
j
\end{matrix}
\right ](a/b)^{j}\\
\times
 \sum_{l\geq 0}
 \left [
\begin{matrix}
N-1-j-l\\
n
\end{matrix}
\right ]
 q^{l(l-1)/2}
 \left (
\frac{cq}{bd}
 \right )^{l}
\left [
\begin{matrix}
n\\
l
\end{matrix}
\right ].
\end{multline}
For $N\geq 2$,
\begin{multline}\label{BNeq}
 B_{N}= A_{N}+b^{N-1}(cq/b-a)\sum_{n\geq 0} (d/b)^{n}q^{n(n+3)/2}
 \sum_{j\geq 0}
\left [
\begin{matrix}
n+j\\
j
\end{matrix}
\right ](a/b)^{j}\\
\times
 \sum_{l\geq 0}
 \left [
\begin{matrix}
N-2-j-l\\
n
\end{matrix}
\right ]
 q^{l(l-1)/2}
 \left (
\frac{cq}{bd}
 \right )^{l}
\left [
\begin{matrix}
n\\
l
\end{matrix}
\right ].
\end{multline}
\end{theorem}

We now prove the main result of this section.

\begin{theorem}\label{t1ef}
Let $x,y,z$ and $q$ be complex numbers with $|q| <1$ and $y\not =
q^{-n}$, $y\not = x z q^{n}$ for any integer $n \geq 0$. Let
$\phi(x,y,z,q)$ be defined by
\begin{equation}\label{phieq1}
\phi(x,y,z,q):= \sum_{n=0}^{\infty}
\frac{x^{n}q^{n(n+1)/2}(-z;q)_{n}} {(y;q)_{n+1}(q;q)_{n}},
\end{equation}
and for  each positive integer $m$, define $e_m(x,y,z,q)$  by
{\allowdisplaybreaks
\begin{align}
e_m(x,y,z,q)&:=\sum_{n\geq 0} x^{n}q^{n(n+1)/2}
 \sum_{j\geq 0}
\left [
\begin{matrix}
n+j\\
j
\end{matrix}
\right ]y^{j}\\
&\phantom{sadasdasd}\times
 \sum_{l\geq 0}
 \left [
\begin{matrix}
m-1-j-l\\
n
\end{matrix}
\right ]
 q^{l(l-1)/2}
 z^{l}
\left [
\begin{matrix}
n\\
l
\end{matrix}
\right ].\notag
\end{align}}
Then, for $m \geq 2$,
\begin{equation}\label{phirecureq1}
\phi(x q^m,y,z,q)=\frac{e_m(x,y,z,q)\phi(x
q,y,z,q)-e_{m-1}(xq,y,z,q)\phi(x,y,z,q)} {\prod_{j=1}^{m-1}(y-x z
q^j)}.
\end{equation}
\end{theorem}
Remarks: For ease in following the proof below, define
\begin{align*}
f_m(x,y,z,q)&:=\sum_{n\geq 0} x^{n}q^{n(n+3)/2}
 \sum_{j\geq 0}
\left [
\begin{matrix}
n+j\\
j
\end{matrix}
\right ]y^{j}\\
&\phantom{sadasdasd}\times
 \sum_{l\geq 0}
 \left [
\begin{matrix}
m-2-j-l\\
n
\end{matrix}
\right ]
 q^{l(l-1)/2}
 z^{l}
\left [
\begin{matrix}
n\\
l
\end{matrix}
\right ]\notag\\
&=e_{m-1}(xq,y,z,q).\notag
\end{align*}
  Note also that Lemma \ref{qbin} gives, for $|y|<1$, that
\begin{equation}\label{eflim}
\lim_{m\to \infty}e_m(x,y,z,q)=\phi(x,y,z,q). \notag
\end{equation}
Also for ease of notation in the proof below, we define
\begin{align*}
e_{n,m}&:=e_n(xq^m,y,z,q),\\
f_{n,m}&:=f_n(xq^m,y,z,q)
\end{align*}
\begin{proof}[Proof of Theorem \ref{t1ef}]
From Theorem \ref{qth}, the $n$-th numerator convergent of the continued fraction
\begin{equation}\label{cfeq}
\frac{-y+zx}{y+1+xq}
\+\frac{-y+zxq}{y+1+xq^2} \+ \cds \+ \frac{-y+zxq^{n-1}}{y+1+xq^n}\+
\cds
\end{equation}
is $(zx-y)f_{n+1,0}$, and  the $n$-th denominator convergent is $e_{n+1,0}$, upon noting that
\begin{align}\label{Aneneq}
A_{n+1}&=e_{n+1}(x,y,z,q)= e_{n+1,0},\\
B_{n+1}-A_{n+1}&=(zx-y)f_{n+1}(x,y,z,q)=(zx-y)f_{n+1,0}.
\end{align}
Thus {\allowdisplaybreaks
\begin{multline*}
\frac{(zx-y)f_{n+m+1,0}}{e_{n+m+1,0}}
=\frac{-y+zx}{y+1+xq}\+\cds\+\frac{-y+zxq^{m-1}}{y+1+xq^m}
\\
\phantom{asdasdasasasda}\+
\frac{-y+zxq^{m}}{y+1+xq^{m+1}}
\+\cds\+\frac{-y+zxq^{m+n-1}}{y+1+xq^{m+n}}\\
=\frac{-y+zx}{y+1+xq}\+\cds\+\frac{-y+zxq^{m-1}}{y+1+xq^m}
\+
\frac{(zxq^{m}-y)f_{n+1,m}}
{e_{n+1,m}}
\end{multline*}}
Thus, by \eqref{recurrel},
\begin{align}\label{polyver}
(zx-y)f_{n+m+1,0}
&=[e_{n+1,m}][(zx-y)f_{m+1,0}]\\
&\phantom{aasdaaasa}
+[(zxq^{m}-y)f_{n+1,m}][(zx-y)f_{m,0}],\notag\\
e_{n+m+1,0}
&=[e_{n+1,m}][e_{m+1,0}]
+[(zxq^{m}-y)f_{n+1,m}][e_{m,0}].\notag
\end{align}
Next, divide through the first equation by
$zx-y$,  let $n \to \infty$ (here taking $|y|<1$), and use \eqref{eflim} to get
\begin{align}\label{server}
 \phi(xq,y,z,q)&=\phi(xq^m,y,z,q)f_{m+1,0}\\
&\phantom{aa}+
(zxq^{m}-y)\phi(xq^{m+1},y,z,q)f_{m,0}, \notag \\
 \phi(x,y,z,q)&=\phi(xq^m,y,z,q)e_{m+1,0} \notag \\
&\phantom{aa}+ (zxq^{m}-y)\phi(xq^{m+1},y,z,q)e_{m,0}. \notag
\end{align}
We solve this last pair of equations for
$\phi(xq^{m+1},y,z,q)$  and  $\phi(xq^m,y,z,q)$ to get that
\begin{equation*}
\phi(xq^m,y,z,q)=\frac{e_{m,0}\phi(xq,y,z,q)-f_{m,0}\phi(x,y,z,q)}
{f_{m+1,0}e_{m,0}-e_{m+1,0}f_{m,0}},
\end{equation*}
and the result follows for $|y|<1$ upon noting that \eqref{reclem} give
\begin{multline*}
(zx-y)f_{m+1,0}e_{m,0}-e_{m+1,0}(zx-y)f_{m,0}
=(-1)^{m-1}\prod_{i=1}^{m}(-y+zxq^{i-1}).
\end{multline*}
The full result follows from the Identity Theorem, regarding each side of \eqref{phirecureq1} as a function of $y$.
\end{proof}

Remark: Special cases of $e_m(q):=e_m(x,y,z,q)$ (the $a_m(q)$ and $b_m(q)$ in the corollaries below) were initially derived as solutions to various recursions (see, for example, the paper of Sills \cite{S03}). While this recursion is not necessary for our present work, we include it for the sake of completeness. From \eqref{recurrel}, \eqref{cfeq}, \eqref{Aneneq} and Theorem \ref{qth} it is not difficult to see that this recurrence has the form
\begin{equation}\label{emrecur}
e_{m+1}(q)=(y+1+xq^m)e_m(q)+(-y+z xq^{m-1})e_{m-1}(q),
\end{equation}
with $e_0(q)=0$ and $e_1(q)=1$.

As a first application, we give a proof of the result of Garrett, Ismail and Stanton at \eqref{gisrr}.

\begin{corollary}\label{c1}
For $|q|<1$ and integral $m\geq 0$,
\begin{equation}\label{rrgen}
\sum_{n=0}^{\infty}\frac{q^{n^2+m n}}{(q;q)_n}=\frac{(-1)^m q^{-m(m-1)/2}a_m(q)}{(q,q^4;q^5)}+\frac{(-1)^{m+1} q^{-m(m-1)/2}b_m(q)}{(q^2,q^3;q^5)},
\end{equation}
where $a_0(q)=1$, $b_0(q)=0$, and for $m\geq 1$,
\begin{align*}
a_m(q)&=\sum_{n} q^{n^2+n}\left [
\begin{matrix}
m-2-n\\
n
\end{matrix}
\right ],\\
b_m(q)&=\sum_{n} q^{n^2}\left [
\begin{matrix}
m-1-n\\
n
\end{matrix}
\right ].
\end{align*}
\end{corollary}

\begin{proof}
In \eqref{phirecureq1}, set  $z=1/x$ and let $y, x \to 0$. Then
\begin{align*}\lim_{x \to 0}f_m(x,0,1/x,q)&= a_m(q),\\
\lim_{x \to 0}e_m(x,0,1/x, q)&=b_m(q). \end{align*} Likewise, it is
not difficult to see that
\[
\lim_{x \to 0}\phi(xq^m,0,1/x,q)=
\sum_{n=0}^{\infty}\frac{q^{n^2+mn}}{(q;q)_n}.
\]

Use the Rogers-Ramanujan identities to get that
\begin{align}\label{rridseq}
\lim_{x \to 0}\phi(x,0,1/x,q)&=
\sum_{n=0}^{\infty}\frac{q^{n^2}}{(q;q)_n}
=\frac{1}{(q,q^4;q^5)_{\infty}},\\
\lim_{x \to
0}\phi(xq,0,x,q)&=\sum_{n=0}^{\infty}\frac{q^{n^2+n}}{(q;q)_n}
=\frac{1}{(q^2,q^3;q^5)_{\infty}}. \notag
\end{align}
Lastly, with the stated values for the parameters, the denominator on the right side of \eqref{phirecureq1} now becomes  $(-1)^{m-1}q^{m(m-1)/2}$.
\end{proof}

We now prove a number of similar identities, giving explicit
formulae for the polynomials corresponding to the $a_m(q)$ and
$b_m(q)$ in the corollary above.

\begin{corollary}\label{c2}
For $|q|<1$ and integral $m\geq 0$,
\begin{multline}\label{rrgen2}
\sum_{n=0}^{\infty}\frac{(-q;q)_nq^{n(n-1)/2+m
n}}{(q;q)_n}\\
=\frac{(-1)^{m-1}}{q^{m(m-1)/2}}\bigg
[(a_m(q)-b_m(q))\frac{(q^4;q^4)_{\infty}}{(q;q)_{\infty}}
-b_m(q)\frac{(-q;q^2)_{\infty}}{(q;q^2)_{\infty}}\bigg ]  ,
\end{multline}
where $a_0(q)=0$, $b_0(q)=1$, and for $m\geq 1$,
\begin{align*}
a_m(q)&=\sum_{n,l} q^{n(n-1)/2+l(l+1)/2}\left [
\begin{matrix}
m-1-l\\
n
\end{matrix}
\right ] \left [
\begin{matrix}
n\\
l
\end{matrix}
\right ],\\
b_m(q)&=\sum_{n,l} q^{n(n+1)/2+l(l+1)/2}\left [
\begin{matrix}
m-2-l\\
n
\end{matrix}
\right ] \left [
\begin{matrix}
n\\
l
\end{matrix}
\right ].
\end{align*}
\end{corollary}

\begin{proof}
In \eqref{phirecureq1}, set $x=1/q$, $y=0$, $z=q$ and use the
 identities (see \textbf{A.8} and \textbf{A.13} in \cite{S03})
\begin{align*}
\phi(1/q,0,q,q)&=
\sum_{n=0}^{\infty}\frac{(-q;q)_nq^{n(n-1)/2}}{(q;q)_n}
=\frac{(q^4;q^4)_{\infty}}{(q;q)_{\infty}}
+\frac{(-q;q^2)_{\infty}}{(q;q^2)_{\infty}},\\
\phi(1,0,q,q)&=\sum_{n=0}^{\infty}\frac{(-q;q)_nq^{n(n+1)/2}}{(q;q)_n}
=\frac{(q^4;q^4)_{\infty}}{(q;q)_{\infty}}.
\end{align*}
\end{proof}

\begin{corollary}\label{c3}
For $|q|<1$ and integral $m\geq 0$,
\begin{multline}\label{rrgen3}
\sum_{n=0}^{\infty}\frac{q^{n^2+2m
n}}{(q^4;q^4)_n}
\\
=(-1)^{m-1}\bigg
[\frac{a_m(q)}{(q^2,q^3;q^5)_{\infty}(-q^2;q^2)_{\infty}}
-\frac{b_m(q)}{(q,q^4;q^5)_{\infty}(-q^2;q^2)_{\infty}}\bigg ],
\end{multline}
where $a_0(q)=0$, $b_0(q)=1$, and for $m\geq 1$,
\begin{align*}
a_m(q)&=\sum_{n,j} q^{n^2}(-1)^j\left [
\begin{matrix}
m-1-j\\
n
\end{matrix}
\right ]_{q^2} \left [
\begin{matrix}
n+j\\
j
\end{matrix}
\right ]_{q^2},\\
b_m(q)&=\sum_{n,j} q^{n^2+2n}(-1)^j\left [
\begin{matrix}
m-2-j\\
n
\end{matrix}
\right ]_{q^2} \left [
\begin{matrix}
n+j\\
j
\end{matrix}
\right ]_{q^2}.
\end{align*}
\end{corollary}

\begin{proof}
In \eqref{phirecureq1}, replace $q$ with $q^2$, set $x=1/q$, $y=-1$,
$z=0$ and use the
 identities (see \textbf{A.16} and \textbf{A.20} in \cite{S03})
\begin{align}\label{c3idseq}
\phi(1/q,-1,0,q^2)&= \sum_{n=0}^{\infty}\frac{q^{n^2}}{2(q^4;q^4)_n}
=\frac{1}{2(q,q^4;q^5)_{\infty}(-q^2;q^2)_{\infty}},\\
\phi(q,-1,0,q^2)&=
\sum_{n=0}^{\infty}\frac{q^{n^2+2n}}{2(q^4;q^4)_n}
=\frac{1}{2(q,q^4;q^5)_{\infty}(-q^2;q^2)_{\infty}}.\notag
\end{align}
The result follows after cancelling the ``2" factor in the
denominators.
\end{proof}

The next corollary involves the analytic versions of the
G\"{o}llnitz-Gordon identities.

\begin{corollary}\label{c4}
For $|q|<1$ and integral $m\geq 0$,
\begin{multline}\label{rrgen4}
\sum_{n=0}^{\infty}\frac{(-q;q^2)_nq^{n^2+2m n}}{(q^2;q^2)_n}
=\frac{(-1)^{m-1}}{q^{m(m-1)}}\bigg
[\frac{a_m(q)}{(q^3,q^4,q^5;q^8)_{\infty}}
-\frac{b_m(q)}{(q,q^4,q^7;q^8)_{\infty}}\bigg ],
\end{multline}
where $a_0(q)=0$, $b_0(q)=1$, and for $m\geq 1$,
\begin{align*}
a_m(q)&=\sum_{n,l} q^{n^2+l^2}\left [
\begin{matrix}
m-1-l\\
n
\end{matrix}
\right ]_{q^2} \left [
\begin{matrix}
n\\
l
\end{matrix}
\right ]_{q^2},\\
b_m(q)&=\sum_{n,l} q^{n^2+2n+l^2}\left [
\begin{matrix}
m-2-l\\
n
\end{matrix}
\right ]_{q^2} \left [
\begin{matrix}
n\\
l
\end{matrix}
\right ]_{q^2}.
\end{align*}
\end{corollary}

\begin{proof}
In \eqref{phirecureq1}, replace $q$ with $q^2$, set $x=1/q$, $y=0$,
$z=q$ and use the
 identities (see \textbf{A.34} and \textbf{A.36} in \cite{S03})
\begin{align}\label{ggidseq}
\phi(1/q,0,q,q^2)&=
\sum_{n=0}^{\infty}\frac{(-q;q^2)_nq^{n^2}}{(q^2;q^2)_n}
=\frac{1}{(q,q^4,q^7;q^8)_{\infty}},\\
\phi(q,0,q,q^2)&=
\sum_{n=0}^{\infty}\frac{(-q;q^2)_nq^{n^2+2n}}{(q^2;q^2)_n}
=\frac{1}{(q^3,q^4,q^5;q^8)_{\infty}}. \notag
\end{align}
\end{proof}

\subsection{Implications of Watson's Transformation}

We next recall Watson's transformation:
\begin{multline*}
_{8} \phi _{7} \left (
\begin{matrix}
A,\,q\sqrt{A},\,-q\sqrt{A},\,B,\,C,\,D,\,E,\,q^{-n}\\
\sqrt{A},\,-\sqrt{A},\,Aq/B,\,Aq/C,\,Aq/D,\,Aq/E,\,Aq^{n+1}\,
\end{matrix}
; q,\frac{A^{2}q^{n+2}}{BCDE} \right ) = \\
 \frac{(A q)_{n} (A q/DE)_{n}  }{(A q/D)_{n} (A q/E)_{n}
}   \,\,   _{4}\phi _{3} \left (
\begin{matrix}
Aq/BC,D,E,q^{-n}\\
Aq/B,Aq/C,DEq^{-n}/A
\end{matrix}\,
; q,q \right ),
\end{multline*}
where $n$ is a non-negative integer. If we let $B$, $D$ and $n \to
\infty$ (as in \cite{H74}), replace $A$ with $zx$, $C$ with $zx/y$
and $E$ with $-z$, and multiply both sides by $1/(1-y)$, we get
\begin{multline}\label{wateq1}
\sum_{n \geq 0} \frac{(1-zxq^{2n})(zx,zx/y,-z;q)_{n}\left( xy\right
)^{n} q^{n(3n+1)/2}}
{(1-zx)(-xq,q;q)_{n}(y;q)_{n+1}}\\
= \frac{(zxq;q)_{\infty}}{(-xq;q)_{\infty}} \sum_{n \geq 0}\frac{
 x^{n}q^{n(n+1)/2}(-z;q)_{n}}{(y;q)_{n+1}(q;q)_{n}}.
\end{multline}

Notice that the series on the right is the series $\phi(x,y,z,q)$
from \eqref{phieq1}, so that the special case of Watson's
transformation at \eqref{wateq1} may be used in conjunction with the
specializations of $x$, $y$ and $z$ in Corollaries \ref{c1}-\ref{c4} to produce a new set of summation formulae.

\begin{corollary}\label{c1w}
For $|q|<1$ and integral $m\geq 0$,
\begin{multline}\label{rrgenw}
\sum_{n=0}^{\infty}\frac{ (1-q^{2n+m})(q;q)_{n+m-1} q^{n(5n-1)/2+2m
n}(-1)^n}{(q;q)_n}\\
=
\frac{b_m(q)(q,q^4,q^5;q^5)_{\infty}-a_m(q)(q^2,q^3,q^5;q^5)_{\infty}}
{(-1)^{m-1}q^{m(m-1)/2}},
\end{multline}
where $a_0(q)=1$, $b_0(q)=0$, and for $m\geq 1$,
\begin{align*}
a_m(q)&=\sum_{n} q^{n^2+n}\left [
\begin{matrix}
m-2-n\\
n
\end{matrix}
\right ],\\
b_m(q)&=\sum_{n} q^{n^2}\left [
\begin{matrix}
m-1-n\\
n
\end{matrix}
\right ].
\end{align*}
\end{corollary}

\begin{proof}
In \eqref{wateq1}, replace $x$ with $xq^m$, set  $z=1/x$ and then
let $y, x \to 0$. Combine the resulting identity with \eqref{rrgen},
and \eqref{rrgenw} follows.
\end{proof}

\begin{corollary}\label{c2w}
For $|q|<1$ and integral $m\geq 0$,
\begin{multline}\label{rrgen2w}
\sum_{n=0}^{\infty}\frac{(1-q^{2n+m})(-q;q)_{n}(q;q)_{n+m-1}q^{2n^2-n+2m
n}(-1)^n}{(q;q)_{n}(-q;q)_{n+m-1}}\\
=\frac{(-1)^{m-1}}{q^{m(m-1)/2}}\bigg
[(a_m(q)-b_m(q))\frac{(q^4;q^4)_{\infty}}{(-q;q)_{\infty}}
-b_m(q)\frac{(q^2;q^2)_{\infty}}{(-q^2;q^2)_{\infty}}\bigg ],
\end{multline}
where $a_0(q)=0$, $b_0(q)=1$, and for $m\geq 1$,
\begin{align*}
a_m(q)&=\sum_{n,l} q^{n(n-1)/2+l(l+1)/2}\left [
\begin{matrix}
m-1-l\\
n
\end{matrix}
\right ] \left [
\begin{matrix}
n\\
l
\end{matrix}
\right ],\\
b_m(q)&=\sum_{n,l} q^{n(n+1)/2+l(l+1)/2}\left [
\begin{matrix}
m-2-l\\
n
\end{matrix}
\right ] \left [
\begin{matrix}
n\\
l
\end{matrix}
\right ].
\end{align*}
\end{corollary}

\begin{proof}
This time in \eqref{wateq1}, replace $x$ with $xq^m$, and then set
$x=1/q$, $z=q$ and  let $y \to 0$. Combine the resulting identity
with \eqref{rrgen2}, and \eqref{rrgen2w} follows.
\end{proof}

\begin{corollary}\label{c3w}
For $|q|<1$ and integral $m\geq 0$,
\begin{multline}\label{rrgen3w}
\sum_{n=0}^{\infty}\frac{q^{3n^2+2m
n}(-1)^n}{(-q;q^2)_{m+n}(q^4;q^4)_n}\\
=(-1)^{m-1}\bigg
[\frac{a_m(q)}{(q^2,q^3;q^5)_{\infty}(-q;q)_{\infty}}
-\frac{b_m(q)}{(q,q^4;q^5)_{\infty}(-q;q)_{\infty}}\bigg ],
\end{multline}
where $a_0(q)=0$, $b_0(q)=1$, and for $m\geq 1$,
\begin{align*}
a_m(q)&=\sum_{n,j} q^{n^2}(-1)^j\left [
\begin{matrix}
m-1-j\\
n
\end{matrix}
\right ]_{q^2} \left [
\begin{matrix}
n+j\\
j
\end{matrix}
\right ]_{q^2},\\
b_m(q)&=\sum_{n,j} q^{n^2+2n}(-1)^j\left [
\begin{matrix}
m-2-j\\
n
\end{matrix}
\right ]_{q^2} \left [
\begin{matrix}
n+j\\
j
\end{matrix}
\right ]_{q^2}.
\end{align*}
\end{corollary}

\begin{proof}
In \eqref{wateq1}, replace $q$ with $q^2$ and $x$ with $xq^{2m}$,
and then set $x=1/q$, $y=-1$ and   $z = 0$. Combine the resulting
identity with \eqref{rrgen3}, and now \eqref{rrgen3w} follows.
\end{proof}

\begin{corollary}\label{c4w}
For $|q|<1$ and integral $m\geq 0$,
\begin{multline}\label{rrgen4w}
\sum_{n=0}^{\infty}\frac{(1-q^{4n+2m})(-q;q^2)_{n}(q^{2};q^2)_{n+m-1}(-1)^nq^{4n^2-n+4m
n}}{(q^{2};q^2)_{n}(-q;q^2)_{n+m}}\\
=\frac{(-1)^{m-1}}{q^{m(m-1)}}\bigg [a_m(q)(q,q^7,q^8;q^8)_{\infty}
-b_m(q)(q^3,q^5,q^8;q^8)_{\infty}\bigg ],
\end{multline}
where $a_0(q)=0$, $b_0(q)=1$, and for $m\geq 1$,
\begin{align*}
a_m(q)&=\sum_{n,l} q^{n^2+l^2}\left [
\begin{matrix}
m-1-l\\
n
\end{matrix}
\right ]_{q^2} \left [
\begin{matrix}
n\\
l
\end{matrix}
\right ]_{q^2},\\
b_m(q)&=\sum_{n,l} q^{n^2+2n+l^2}\left [
\begin{matrix}
m-2-l\\
n
\end{matrix}
\right ]_{q^2} \left [
\begin{matrix}
n\\
l
\end{matrix}
\right ]_{q^2}.
\end{align*}
\end{corollary}

\begin{proof}
This time in \eqref{wateq1}, replace $q$ with $q^2$ and $x$ with
$xq^{2m}$, and then set $x=1/q$, $y=0$ and   $z = q$. Combine the
resulting identity with \eqref{rrgen4}, and  \eqref{rrgen4w} follows
after some simple $q$-product manipulations.
\end{proof}

\subsection{Implications of Heine's Transformation} A number of other
transformations may be employed to derive new summation formulae, in
ways that are similar to how Watson's transformation was used above.

The first of these is Heine's transformation:
\begin{equation*}
\sum_{n=0}^{\infty}\frac{(a,b;q)_n t^n}{(c,q;q)_n}=
\frac{(b,at;q)_{\infty}}{(c,t;q)_{\infty}}
\sum_{n=0}^{\infty}\frac{(c/b,t;q)_n b^n}{(at,q;q)_n}.
\end{equation*}
If we replace $a$ with $-xq/t$, $b$ with $-z$ and $c$ with $yq$,
multiple both sides by $1/(1-y)$ and let $t \to 0$, then the
following transformation results.
\begin{equation}\label{H1eq}
\sum_{n \geq 0}\frac{
 x^{n}q^{n(n+1)/2}(-z;q)_{n}}{(y;q)_{n+1}(q;q)_{n}}
 =\frac{(-z,-xq;q)_{\infty}}{(y;q)_{\infty}}\sum_{n \geq 0}\frac{
 (-qy/z;q)_{n}(-z)^{n}}{(-xq;q)_{n}(q;q)_{n}}.
\end{equation}

Note that the series on the left is the series $\phi(x,y,z,q)$ from
\eqref{phieq1}, so this transformation may be used in conjunction
with Corollaries \ref{c2}-\ref{c4} to produce new summation
formulae.

\begin{corollary}\label{c2h}
For $|q|<1$ and integral $m\geq 0$,
\begin{equation}\label{rrgen2h}
\sum_{n=0}^{\infty}\frac{(-q)^{n}}{(-q;q)_{m+n-1}(q;q)_n}
=\frac{(-1)^{m-1}}{q^{m(m-1)/2}}\bigg
[\frac{(a_m(q)-b_m(q))}{(-q;q^2)_{\infty}}
-\frac{b_m(q)}{(-q^2;q^2)_{\infty}}\bigg ],
\end{equation}
where  $a_0(q)=0$, $b_0(q)=1$, and for $m\geq 1$,
\begin{align*}
a_m(q)&=\sum_{n,l} q^{n(n-1)/2+l(l+1)/2}\left [
\begin{matrix}
m-1-l\\
n
\end{matrix}
\right ] \left [
\begin{matrix}
n\\
l
\end{matrix}
\right ],\\
b_m(q)&=\sum_{n,l} q^{n(n+1)/2+l(l+1)/2}\left [
\begin{matrix}
m-2-l\\
n
\end{matrix}
\right ] \left [
\begin{matrix}
n\\
l
\end{matrix}
\right ].
\end{align*}
\end{corollary}
\begin{proof}
In \eqref{H1eq},  set $x=q^{m-1}$ and $y=0$, $z =q$, and combine
with \eqref{rrgen2}.
\end{proof}

\begin{corollary}\label{c3h}
For $|q|<1$ and integral $m\geq 0$,
\begin{multline}\label{rrgen3h}
\sum_{n=0}^{\infty}\frac{q^{n^2+
n}}{(-q;q^2)_{m+n}(q^2;q^2)_n}\\
=(-1)^{m-1}\bigg
[\frac{a_m(q)}{(q^2,q^3;q^5)_{\infty}(-q;q^2)_{\infty}}
-\frac{b_m(q)}{(q,q^4;q^5)_{\infty}(-q;q^2)_{\infty}}\bigg ],
\end{multline}
where $a_0(q)=0$, $b_0(q)=1$, and for $m\geq 1$,
\begin{align*}
a_m(q)&=\sum_{n,j} q^{n^2}(-1)^j\left [
\begin{matrix}
m-1-j\\
n
\end{matrix}
\right ]_{q^2} \left [
\begin{matrix}
n+j\\
j
\end{matrix}
\right ]_{q^2},\\
b_m(q)&=\sum_{n,j} q^{n^2+2n}(-1)^j\left [
\begin{matrix}
m-2-j\\
n
\end{matrix}
\right ]_{q^2} \left [
\begin{matrix}
n+j\\
j
\end{matrix}
\right ]_{q^2}.
\end{align*}
\end{corollary}
\begin{proof}
In \eqref{H1eq}, replace $q$ with $q^2$, set $x=q^{2m-1}$ and
$y=-1$, let $z \to 0$ and combine with \eqref{rrgen3}.
\end{proof}

\begin{corollary}\label{c4h}
For $|q|<1$ and integral $m\geq 0$,
\begin{multline}\label{rrgen4h}
\sum_{n=0}^{\infty}\frac{(-q)^{n}}{(-q;q^2)_{m+n}(q^2;q^2)_n}\\
=\frac{(-1)^{m-1}}{q^{m(m-1)}(-q,-q;q^2)_{\infty}}\bigg
[\frac{a_m(q)}{(q^3,q^4,q^5;q^8)_{\infty}}
-\frac{b_m(q)}{(q,q^4,q^7;q^8)_{\infty}}\bigg ],
\end{multline}
where $a_0(q)=0$, $b_0(q)=1$, and for $m\geq 1$,
\begin{align*}
a_m(q)&=\sum_{n,l} q^{n^2+l^2}\left [
\begin{matrix}
m-1-l\\
n
\end{matrix}
\right ]_{q^2} \left [
\begin{matrix}
n\\
l
\end{matrix}
\right ]_{q^2},\\
b_m(q)&=\sum_{n,l} q^{n^2+2n+l^2}\left [
\begin{matrix}
m-2-l\\
n
\end{matrix}
\right ]_{q^2} \left [
\begin{matrix}
n\\
l
\end{matrix}
\right ]_{q^2}.
\end{align*}
\end{corollary}
\begin{proof}
In \eqref{H1eq}, replace $q$ with $q^2$, set $x=q^{2m-1}$ and $y=0$,
 $z=q$, and combine with \eqref{rrgen4}.
\end{proof}

\subsection{Implications of a Transformation of Ramanujan}
Another transformation we consider is one  stated by Ramanujan
(\textbf{Entry 2.2.3} in Chapter 2 of \cite{AB09}), which also
follows as a consequence of a transformation of Sears \cite{S51}:
\begin{equation*}
\sum_{n \geq 0}\frac{
 b^{n}q^{n(n+1)/2}(-aq/b;q)_{n}}{(-cq;q)_{n}(q;q)_{n}}
 =(-bq;q)_{\infty}\sum_{n \geq 0}\frac{
 a^{n}q^{n(n+1)}(bc/a;q)_{n}}{(-bq;q)_{n}(-cq;q)_n(q;q)_{n}}.
\end{equation*}

If we replace $b$ with $x$, $c$ with $-y$ and $a$ with $zx/q$, and
multiply both sides by $1/(1-y)$, then we get
\begin{equation}\label{R2eq}
\sum_{n \geq 0}\frac{
 x^{n}q^{n(n+1)/2}(-z;q)_{n}}{(y;q)_{n+1}(q;q)_{n}}
 =(-xq;q)_{\infty}\sum_{n \geq 0}\frac{
 (zx)^{n}q^{n^2}(-qy/z;q)_{n}}{(-xq;q)_{n}(y;q)_{n+1}(q;q)_{n}}.
\end{equation}

Once again, the series on the left is the series $\phi(x,y,z,q)$
from \eqref{phieq1}, and specializing $x$, $y$ and $z$ as in
Corollaries \ref{c1}-\ref{c4} will give summation formulae similar
to those above (although not all are new).

As an example of an application of this transformation, if we set
$x=q^{m-1}$, $y=0$ and $z=q$, and combine with the identity at
\eqref{rrgen2}, then the following summation formula results.

\begin{corollary}\label{c2m2}
For $|q|<1$ and integral $m\geq 0$,
\begin{multline}\label{rrgenm22}
\sum_{n=0}^{\infty}\frac{q^{n^2+m
n}}{(-q;q)_{m+n-1}(q;q)_n}
\\
=\frac{(-1)^{m-1}}{q^{m(m-1)/2}}\bigg
[(a_m(q)-b_m(q))(-q^2;q^2)_{\infty} -b_m(q)(-q;q^2)_{\infty}\bigg ],
\end{multline}
where $a_0(q)=0$, $b_0(q)=1$, and for $m\geq 1$,
\begin{align*}
a_m(q)&=\sum_{n,l} q^{n(n-1)/2+l(l+1)/2}\left [
\begin{matrix}
m-1-l\\
n
\end{matrix}
\right ] \left [
\begin{matrix}
n\\
l
\end{matrix}
\right ],\\
b_m(q)&=\sum_{n,l} q^{n(n+1)/2+l(l+1)/2}\left [
\begin{matrix}
m-2-l\\
n
\end{matrix}
\right ] \left [
\begin{matrix}
n\\
l
\end{matrix}
\right ].
\end{align*}
\end{corollary}

\section{A Second General Transformation}

The following result was also proven in  \cite{BMcLW06}:
\begin{theorem}\label{1/qth}
Let $a$, $b$, $c$, $d$ be complex numbers with $d \not = 0$ and
$|q|<1$. Define \[ H_{1}(a,b,c,d,q):= \frac{1}{1} \+
\frac{-abq+c}{(a+b)q+d}
 \+ \cds \+ \frac{-ab
q^{2n+1}+cq^n}{(a+b) q^{n+1}+d}\+ \cds .
\]
 Let $C_{N}:=C_{N}(q)$ and $D_{N}:=D_{N}(q)$ denote the $N$-th
numerator convergent and $N$-th denominator convergent,
respectively, of $H_{1}(a,b,c,d,q)$. Then $C_{N}$ and $D_{N}$ are
given explicitly by the following formulae.
\begin{multline}\label{CNeq}
C_{N}= d^{N-1}\sum_{j,\,l ,\,n\geq 0}a^{j}b^{n-j-l}c^{l}
d^{-n-l}q^{n(n+1)/2+l(l-1)/2} \\
 \times\left [
\begin{matrix}
N-1-n+j\\
j
\end{matrix}
\right ]_{q} \left [
\begin{matrix}
N-1-j-l\\
n-j-l
\end{matrix}
\right ]_{q} \left [
\begin{matrix}
N-1-n\\
l
\end{matrix}
\right ]_{q}.
\end{multline}
For $N\geq 2$,
\begin{multline}\label{DNeq}
 D_{N}=C_{N}+(c/bq-a)\sum_{j,\,l,  \, n \geq 0}a^{j}b^{n+1-j-l}c^{l}
d^{N-2-n-l} \times \\
q^{(n+1)(n+2)/2+l(l-1)/2} \left [
\begin{matrix}
N-2-n+j\\
j
\end{matrix}
\right ]_{q} \left [
\begin{matrix}
N-2-j-l\\
n-j-l
\end{matrix}
\right ]_{q} \left [
\begin{matrix}
N-2-n\\
l
\end{matrix}
\right ]_{q}.
 \end{multline}
{\allowdisplaybreaks
\begin{align*}
\lim_{N \to \infty} \frac{C_{N}}{d^{N-1}}
&=(-aq/d)_{\infty} \sum_{j
=0}^{\infty}\frac{(b/d)^{j}(-c/bd)_{j}\,q^{j(j+1)/2}}{(q)_{j}(-aq/d)_{j}}
,\\
\lim_{N \to \infty} \frac{D_{N}-C_{N}}{d^{N-1}}
&=\frac{c-abq}{d}(-aq^2/d)_{\infty} \sum_{j
=0}^{\infty}\frac{(b/d)^{j}(-c/bd)_{j}\,q^{(j^2+3j)/2}}{(q)_{j}(-aq^2/d)_{j}}
.
\end{align*}
}
\end{theorem}

From the result above we derive the following general identity.

\begin{theorem}\label{t2ef}
Let $x,y,z$ and $q$ be complex numbers with $|q| <1$ and $y\not =
-q^{-n}/x$, $y\not =  z q^{-n}/x$ for $n \geq 1$. Let
$\Phi(x,y,z,q)$ be defined  by
\begin{equation}\label{phieq2}
\Phi(x,y,z,q):= \sum_{n=0}^{\infty}
\frac{x^{n}q^{n(n+1)/2}(-z;q)_{n}} {(-xyq;q)_{n}(q;q)_{n}},
\end{equation}
and for a positive integer $m$, define $g_m(x,y,z,q)$  by
{\allowdisplaybreaks
\begin{align}
g_m(x,y,z,q)&:=\sum_{n,j,l\geq 0}
x^{n}y^{j}z^{l}q^{n(n+1)/2+l(l-1)/2} \left [
\begin{matrix}
m-1-n+j\\
j
\end{matrix}
\right ]\\
&\phantom{sadasdaasdsdaadsd}\times
 \left [
\begin{matrix}
m-1-j-l\\
n-j-l
\end{matrix}
\right ] \left [
\begin{matrix}
m-1-n\\
l
\end{matrix}
\right ]. \notag
\end{align}}
Then, for $m \geq 2$,
\begin{multline}\label{phirecureq2}
\sum_{n=0}^{\infty} \frac{x^{n}q^{n(n+1)/2+mn}(-z;q)_{n}}
{(-xyq;q)_{m+n}(q;q)_{n}}=\frac{\Phi(x q^m,y,z,q)}{(-xyq;q)_{m}}=
\\
\frac{g_m(x,y,z,q)\Phi(x
q,y,z,q)/(1+xyq)-g_{m-1}(xq,y,z,q)\Phi(x,y,z,q)}
{\prod_{j=1}^{m-1}(x^2yq^{2j+1}-x z q^j)}.
\end{multline}
\end{theorem}

\begin{proof}
As in the proof of Theorem \ref{t1ef}, for ease of notation we  define a second polynomial sequence,
\begin{align*}
h_m(x,y,z,q)&:=\sum_{n,j,l\geq 0}
x^{n}y^{j}z^{l}q^{n(n+3)/2+l(l-1)/2} \left [
\begin{matrix}
m-2-n+j\\
j
\end{matrix}
\right ]\notag\\
&\phantom{sadasdaasdsdaadsd}\times
 \left [
\begin{matrix}
m-2-j-l\\
n-j-l
\end{matrix}
\right ] \left [
\begin{matrix}
m-2-n\\
l
\end{matrix}
\right ] \notag\\
&=g_{m-1}(xq,y,z,q),
\end{align*}
for integral $m\geq 1$. Note that Lemma \ref{qbin} (and Theorem
\ref{1/qth} - see \cite{BMcLW06} for details) gives  that
\begin{align}\label{eflim2}
\lim_{m\to \infty}g_m(x,y,z,q)&=(-xyq;q)_{\infty}\Phi(x,y,z,q),\\
\lim_{m\to \infty}h_m(x,y,z,q)&=(-xyq^2;q)_{\infty}\Phi(xq,y,z,q).
\notag
\end{align}

For ease of notation, we once again define
\begin{align*}
g_{n,m}&:=g_n(xq^m,y,z,q),\\
h_{n,m}&:=h_n(xq^m,y,z,q)
\end{align*}

From Theorem \ref{1/qth}, the $n$-th numerator convergent of the
continued fraction
\[
\frac{-x^2yq+zx}{(x+xy)q+1} \+\frac{-x^2yq^3+zxq}{(x+xy)q^2+1} \+
\cds \+ \frac{-x^2yq^{2n-1}+zxq^{n-1}}{(x+xy)q^n+1}\+ \cds
\]
is $(zx-x^2yq)h_{n+1,0}$, and  the $n$-th denominator convergent is
$g_{n+1,0}$.

Once again appealing to the recurrence relations for a continued
fraction (see \eqref{recurrel}, and the proof of Theorem
\ref{t1ef}), we have that
\begin{align}\label{polyver2}
&(zx-x^2yq)h_{n+m+1,0}
=[g_{n+1,m}][(zx-x^2yq)h_{m+1,0}]\\
&\phantom{aaasdsasdsdaas}
+[(zxq^{m}-x^2yq^{2m+1})h_{n+1,m}][(zx-x^2yq)h_{m,0}],\notag\\
&g_{n+m+1,0}
=[g_{n+1,m}][g_{m+1,0}]+[(zxq^{m}-x^2yq^{2m+1})h_{n+1,m}][g_{m,0}].\notag
\end{align}
Divide through the first equation by $zx-x^2yq$,  let $n \to \infty$
(here taking $|y|<1$), and use \eqref{eflim2} to get (after dividing
through in each case by $(-xyq^{m+1};q)_{\infty}$) that
\begin{align}\label{server2}
 (-xyq^2;q)_{m-1}\Phi(xq,y,&z,q)=\Phi(xq^m,y,z,q)h_{m+1,0}\\
&+
\frac{zxq^{m}-x^2yq^{2m+1}}{1+xyq^{m+1}}\,\Phi(xq^{m+1},y,z,q)h_{m,0}, \notag \\
 (-xyq;q)_m \Phi(x,y,&z,q)=\Phi(xq^m,y,z,q)g_{m+1,0} \notag \\
&+
\frac{zxq^{m}-x^2yq^{2m+1}}{1+xyq^{m+1}}\,\Phi(xq^{m+1},y,z,q)g_{m,0}.
\notag
\end{align}

Solve this latter pair of equations for $\Phi(xq^m,y,z,q)$ and
$\Phi(xq^{m+1},y,z,q)$ and, as in the proof of Theorem \ref{t1ef}, the result once again follows.
\end{proof}

As with Theorem \ref{t1ef}, Theorem \ref{t2ef} also leads to
summation formulae that are similar to that of Garrett, Ismail and
Stanton at \eqref{gisrr}.

\begin{corollary}\label{cc1}
For $|q|<1$ and integral $m\geq 0$,
\begin{multline}\label{rrgen11}
\sum_{n=0}^{\infty} \frac{q^{n^2+2m n}}{(q;q^2)_{m+n}(q^2;q^2)_n}=\\
\frac{(-1)^{m-1}(-q;q^2)_{\infty}}{q^{2m(m-1)}(q^2;q^2)_{\infty}}\bigg[a_m(q)(q^4,q^{16},q^{20};q^{20})_{\infty}
-b_m(q)(q^8,q^{12},q^{20};q^{20})_{\infty}\bigg],
\end{multline}
where $a_0(q)=0$, $b_0(q)=1$, and for $m\geq 1$,
\begin{align*}
a_m(q)&=\sum_{n,j} q^{n^2}(-1)^j \left [
\begin{matrix}
m-1-n+j\\
j
\end{matrix}
\right ]_{q^2} \left [
\begin{matrix}
m-1-j\\
n-j
\end{matrix}
\right ]_{q^2},\\
b_m(q)&=\sum_{n,j} q^{n^2+2n}(-1)^j \left [
\begin{matrix}
m-2-n+j\\
j
\end{matrix}
\right ]_{q^2} \left [
\begin{matrix}
m-2-j\\
n-j
\end{matrix}
\right ]_{q^2}.
\end{align*}
\end{corollary}

\begin{proof}
In \eqref{phirecureq2}, replace $q$ with $q^2$, set $x=1/q$, $y=-1$,
$z=0$ and use the
 identities (see \textbf{A.79} and \textbf{A.96} in \cite{S03})
\begin{align}\label{a7996eq}
&\Phi(1/q,-1,0,q^2)=
\sum_{n=0}^{\infty}\frac{q^{n^2}}{(q,q^2;q^2)_n}
=(q^8,q^{12},q^{20};q^{20})_{\infty}\frac{(-q;q^2)_{\infty}}
{(q^2;q^2)_{\infty}},\\
\Phi(q,-1,&0,q^2)=
\sum_{n=0}^{\infty}\frac{q^{n^2+2n}}{(q^3,q^2;q^2)_n}
=(1-q)(q^4,q^{16},q^{20};q^{20})_{\infty}\frac{(-q;q^2)_{\infty}}
{(q^2;q^2)_{\infty}}.\notag
\end{align}
The result follows after some simple algebraic manipulations.
\end{proof}

\begin{corollary}\label{cc2}
For $|q|<1$ and integral $m\geq 0$,
\begin{multline}\label{rrgen22}
\sum_{n=0}^{\infty} \frac{q^{2n^2+2m n}}{(q;q^2)_{m+n}(q^2;q^2)_n}\\
=\frac{(-1)^{m-1}}{q^{m(m-1)}(q^2;q^2)_{\infty}}\bigg
[a_m(q)(-q,-q^{7},q^{8};q^{8})_{\infty}
-b_m(q)(-q^3,-q^{5},q^{8};q^{8})_{\infty}\bigg ],
\end{multline}
where $a_0(q)=0$, $b_0(q)=1$, and for $m\geq 1$,
\begin{align*}
a_m(q)&=\sum_{n,l} q^{n^2+l^2}(-1)^{n-l} \left [
\begin{matrix}
m-1-l\\
n-l
\end{matrix}
\right ]_{q^2} \left [
\begin{matrix}
m-1-n\\
l
\end{matrix}
\right ]_{q^2},\\
b_m(q)&=\sum_{n,l} q^{n^2+2n+l^2}(-1)^{n-l} \left [
\begin{matrix}
m-2-l\\
n-l
\end{matrix}
\right ]_{q^2} \left [
\begin{matrix}
m-2-n\\
l
\end{matrix}
\right ]_{q^2}.
\end{align*}
\end{corollary}

\begin{proof}
In \eqref{phirecureq2}, replace $q$ with $q^2$, set $y=-1/(x q)$,
$z=1/x$, let $x \to 0$ and use the
 identities (see \textbf{A.38} and \textbf{A.39} in \cite{S03})
\begin{align*}
\Phi(x,-1/xq,1/x,q^2)&\to
\sum_{n=0}^{\infty}\frac{q^{2n^2}}{(q,q^2;q^2)_n}
=\frac{(-q^3,-q^{5},q^{8};q^{8})_{\infty}}
{(q^2;q^2)_{\infty}},\\
\Phi(xq^2,-1/xq,1/x,q^2)&\to
\sum_{n=0}^{\infty}\frac{q^{2n^2+2n}}{(q^3,q^2;q^2)_n}
=\frac{(1-q)(-q,-q^{7},q^{8};q^{8})_{\infty}} {(q^2;q^2)_{\infty}}.
\end{align*}
\end{proof}

\begin{corollary}\label{cc3}
For $|q|<1$ and integral $m\geq 0$,
\begin{multline}\label{rrgen33}
\sum_{n=0}^{\infty} \frac{(-q;q^2)_{n}q^{n^2+2m n}}
{(q;q^2)_{m+n}(q^2;q^2)_n}
=\frac{(-1)^{m-1}}{q^{m(m-1)}(-q^2;q^2)_{m-1}(q;q)_{\infty}}\\
\times\bigg [a_m(q)(q^2,q^{10},q^{12};q^{12})_{\infty}
-b_m(q)(q^6,q^{6},q^{12};q^{12})_{\infty}\bigg ],
\end{multline}
where $a_0(q)=0$, $b_0(q)=1$, and for $m\geq 1$,
\begin{align*}
a_m(q)&=\sum_{n,j,l\geq 0} (-1)^{j}q^{n^2+l^2} \left [
\begin{matrix}
m-1-n+j\\
j
\end{matrix}
\right ]_{q^2}\\
&\phantom{sadasdaasdsdaadsd}\times
 \left [
\begin{matrix}
m-1-j-l\\
n-j-l
\end{matrix}
\right ]_{q^2} \left [
\begin{matrix}
m-1-n\\
l
\end{matrix}
\right ]_{q^2},\\
b_m(q)&=\sum_{n,j,l\geq 0} (-1)^{j}q^{n^2+2n+l^2} \left [
\begin{matrix}
m-2-n+j\\
j
\end{matrix}
\right ]_{q^2}\\
&\phantom{sadasdaasdsdaadsd}\times
 \left [
\begin{matrix}
m-2-j-l\\
n-j-l
\end{matrix}
\right ]_{q^2} \left [
\begin{matrix}
m-2-n\\
l
\end{matrix}
\right ]_{q^2}.
\end{align*}
\end{corollary}

\begin{proof}
In \eqref{phirecureq2}, replace $q$ with $q^2$, set $x=1/q$, $y=-1$,
$z=q$, and use the
 identities (see \textbf{A.29} and \textbf{A.50} in \cite{S03})
\begin{align*}
\Phi(1/q,-1,q,q^2)&=
\sum_{n=0}^{\infty}\frac{(-q;q^2)_nq^{n^2}}{(q,q^2;q^2)_n}
=\frac{(q^6,q^{6},q^{12};q^{12})_{\infty}}
{(q;q)_{\infty}},\\
\Phi(q,-1,q,q^2)&=
\sum_{n=0}^{\infty}\frac{(-q;q^2)_nq^{n^2+2n}}{(q^3,q^2;q^2)_n}
=\frac{(1-q)(q^2,q^{10},q^{12};q^{12})_{\infty}} {(q;q)_{\infty}}.
\end{align*}
\end{proof}

 The transformations of Watson, Heine and Ramanujan
that were used in conjunction with the series in Theorem \ref{t1ef}
in the previous section to produce new summation formulae may
similarly be used in conjunction with the series in Theorem
\ref{t2ef} (although not all the resulting summations are new). We
first consider Watson's transformation.

\subsection{Watson's Transformation Again}
If we cancel a factor of $1/(1-y)$  in \eqref{wateq1} and then
replace $y$ with $-xy$,  we get the transformation
{\allowdisplaybreaks\begin{multline}\label{wateq2}
\sum_{n \geq 0} \frac{(1-zxq^{2n})(zx,-z/y,-z;q)_{n}\left(
-x^2y\right )^{n} q^{n(3n+1)/2}}
{(1-zx)(-xq,q;q)_{n}(-xyq;q)_{n}}\\
= \frac{(zxq;q)_{\infty}}{(-xq;q)_{\infty}} \sum_{n \geq 0}\frac{
 x^{n}q^{n(n+1)/2}(-z;q)_{n}}{(-xyq;q)_{n}(q;q)_{n}},
\end{multline}}
where the series on the right is the series $\Phi(x,y,z,q)$ from
Theorem \ref{t2ef}.

\begin{corollary}\label{cc1w}
For $|q|<1$ and integral $m\geq 0$,
\begin{multline}\label{rrgen11w}
\sum_{n=0}^{\infty} \frac{q^{3n^2-n+4m n}}{(q^2;q^4)_{m+n}(q^2;q^2)_n}=\\
\frac{(-1)^{m-1}}{q^{2m(m-1)}(q^2;q^2)_{\infty}}\bigg[a_m(q)(q^4,q^{16},q^{20};q^{20})_{\infty}
-b_m(q)(q^8,q^{12},q^{20};q^{20})_{\infty}\bigg],
\end{multline}
where $a_0(q)=0$, $b_0(q)=1$, and for $m\geq 1$,
\begin{align*}
a_m(q)&=\sum_{n,j} q^{n^2}(-1)^j \left [
\begin{matrix}
m-1-n+j\\
j
\end{matrix}
\right ]_{q^2} \left [
\begin{matrix}
m-1-j\\
n-j
\end{matrix}
\right ]_{q^2},\\
b_m(q)&=\sum_{n,j} q^{n^2+n}(-1)^j \left [
\begin{matrix}
m-2-n+j\\
j
\end{matrix}
\right ]_{q^2} \left [
\begin{matrix}
m-2-j\\
n-j
\end{matrix}
\right ]_{q^2}.
\end{align*}
\end{corollary}

\begin{proof}
Replace $q$ with $q^2$ in \eqref{wateq2}, set $x=q^{2m-1}$, $y=-1$,
$z=0$, combine the resulting identity with \eqref{rrgen11}, and the
result follows after some simple $q$-product manipulations.
\end{proof}

\begin{corollary}\label{cc3w}
For $|q|<1$ and integral $m\geq 0$,
\begin{multline}\label{rrgen33w}
\sum_{n=0}^{\infty}
\frac{(1-q^{4n+2m})(q^2;q^4)_{n}(q^2;q^2)_{m+n-1}q^{3n^2-n+4m n}}
{(q^2;q^4)_{m+n}(q^2;q^2)_n}
=\\
\frac{(-1)^{m-1}(-q^2;q^2)_{\infty}}{q^{m(m-1)}(-q^2;q^2)_{m-1}}
\bigg [a_m(q)(q^2,q^{10},q^{12};q^{12})_{\infty}
-b_m(q)(q^6,q^{6},q^{12};q^{12})_{\infty}\bigg ],
\end{multline}
where $a_0(q)=0$, $b_0(q)=1$, and for $m\geq 1$,
{\allowdisplaybreaks\begin{align*}
a_m(q)&=\sum_{n,j,l\geq 0} (-1)^{j}q^{n^2+l^2} \left [
\begin{matrix}
m-1-n+j\\
j
\end{matrix}
\right ]_{q^2}\\
&\phantom{sadasdaasdsdaadsd}\times
 \left [
\begin{matrix}
m-1-j-l\\
n-j-l
\end{matrix}
\right ]_{q^2} \left [
\begin{matrix}
m-1-n\\
l
\end{matrix}
\right ]_{q^2},\\
b_m(q)&=\sum_{n,j,l\geq 0} (-1)^{j}q^{n^2+2n+l^2} \left [
\begin{matrix}
m-2-n+j\\
j
\end{matrix}
\right ]_{q^2}\\
&\phantom{sadasdaasdsdaadsd}\times
 \left [
\begin{matrix}
m-2-j-l\\
n-j-l
\end{matrix}
\right ]_{q^2} \left [
\begin{matrix}
m-2-n\\
l
\end{matrix}
\right ]_{q^2}.
\end{align*}}
\end{corollary}
\begin{proof}
Once again replace $q$ with $q^2$ in \eqref{wateq2}, set
$x=q^{2m-1}$, $y=-1$, $z=q$, and combine the resulting identity with
\eqref{rrgen33}.
\end{proof}

\subsection{Heine's Transformation Again}
If we cancel a factor of $1/(1-y)$ in \eqref{H1eq}, replace $y$ with
$-xy$ and rearrange, we get the identity
\begin{equation}\label{H2eq}
\sum_{n \geq 0}\frac{
 (qxy/z;q)_{n}(-z)^{n}}{(-xq;q)_{n}(q;q)_{n}}
 =\frac{(-xyq;q)_{\infty}}{(-z,-xq;q)_{\infty}}
\sum_{n \geq 0}\frac{
 x^{n}q^{n(n+1)/2}(-z;q)_{n}}{(-xyq;q)_{n}(q;q)_{n}}.
 \end{equation}

We do \emph{not} get a new summation formula from combining
\eqref{H2eq} with Corollary \ref{cc1}, but we do get that if $a(q)$
and $b(q)$ are as defined in Corollary \ref{cc1}, then $a(q)=
a(-q)$, and $b(q)=b(-q)$. This was initially not obvious, but actually follows immediately upon replacing $j$ with $n-j$ in the equations defining  $a(q)$
and $b(q)$.

\begin{corollary}\label{cc3h}
For $|q|<1$ and integral $m\geq 0$,
\begin{multline}\label{rrgen33h}
\sum_{n=0}^{\infty} \frac{(-q^2;q^2)_{m+n-1}q^{n}}
{(q;q^2)_{m+n}(q^2;q^2)_n}
=\frac{(-1)^{m-1}(-q;q)_{\infty}}{q^{m(m-1)}(q;q)_{\infty}}\\
\times\bigg [a_m(q)(q^2,q^{10},q^{12};q^{12})_{\infty}
-b_m(q)(q^6,q^{6},q^{12};q^{12})_{\infty}\bigg ],
\end{multline}
where $a_0(q)=0$, $b_0(q)=1$, and for $m\geq 1$,
\begin{align*}
a_m(q)&=\sum_{n,j,l\geq 0} (-1)^{j+l+n}q^{n^2+l^2} \left [
\begin{matrix}
m-1-n+j\\
j
\end{matrix}
\right ]_{q^2}\\
&\phantom{sadasdaasdsdaadsd}\times
 \left [
\begin{matrix}
m-1-j-l\\
n-j-l
\end{matrix}
\right ]_{q^2} \left [
\begin{matrix}
m-1-n\\
l
\end{matrix}
\right ]_{q^2},\\
b_m(q)&=\sum_{n,j,l\geq 0} (-1)^{j+l+n}q^{n^2+2n+l^2} \left [
\begin{matrix}
m-2-n+j\\
j
\end{matrix}
\right ]_{q^2}\\
&\phantom{sadasdaasdsdaadsd}\times
 \left [
\begin{matrix}
m-2-j-l\\
n-j-l
\end{matrix}
\right ]_{q^2} \left [
\begin{matrix}
m-2-n\\
l
\end{matrix}
\right ]_{q^2}.
\end{align*}
\end{corollary}

\begin{proof}
In \eqref{H2eq}, replace $q$ with $q^2$, set $x=q^{2m-1}$, $y=-1$
and $z=q$. Combine the resulting identity with \eqref{rrgen33}, and
finally replace $q$ with $-q$.
\end{proof}

\subsection{Ramanujan's Transformation Again}
If a factor of $1/(1-y)$ is cancelled in \eqref{R2eq} and then $y$
is replaced with $-xy$ and the resulting equation rearranged, we get
\begin{equation}\label{R2eq2}
\sum_{n \geq 0}\frac{
 (zx)^{n}q^{n^2}(qxy/z;q)_{n}}{(-xq;q)_{n}(-xyq;q)_{n}(q;q)_{n}}
 =\frac{1}{(-xq;q)_{\infty}}
\sum_{n \geq 0}\frac{
 x^{n}q^{n(n+1)/2}(-z;q)_{n}}{(-xyq;q)_{n}(q;q)_{n}}.
\end{equation}

\begin{corollary}\label{cc3r}
For $|q|<1$ and integral $m\geq 0$,
\begin{multline}\label{rrgen33r}
\sum_{n=0}^{\infty} \frac{(-q^2;q^2)_{m+n-1}q^{2n^2+2m n}}
{(q^2;q^4)_{m+n}(q^2;q^2)_n}
=\frac{(-1)^{m-1}(-q^2;q^2)_{\infty}}{q^{m(m-1)}(q^2;q^2)_{\infty}}\\
\times\bigg [a_m(q)(q^2,q^{10},q^{12};q^{12})_{\infty}
-b_m(q)(q^6,q^{6},q^{12};q^{12})_{\infty}\bigg ],
\end{multline}
where $a_0(q)=0$, $b_0(q)=1$, and for $m\geq 1$,

\vspace{10pt}

{\allowdisplaybreaks\begin{align*} a_m(q)&=\sum_{n,j,l\geq 0}
(-1)^{j}q^{n^2+l^2} \left [
\begin{matrix}
m-1-n+j\\
j
\end{matrix}
\right ]_{q^2}\\
&\phantom{sadasdaasdsdaadsd}\times
 \left [
\begin{matrix}
m-1-j-l\\
n-j-l
\end{matrix}
\right ]_{q^2} \left [
\begin{matrix}
m-1-n\\
l
\end{matrix}
\right ]_{q^2},\\
b_m(q)&=\sum_{n,j,l\geq 0} (-1)^{j}q^{n^2+2n+l^2} \left [
\begin{matrix}
m-2-n+j\\
j
\end{matrix}
\right ]_{q^2}\\
&\phantom{sadasdaasdsdaadsd}\times
 \left [
\begin{matrix}
m-2-j-l\\
n-j-l
\end{matrix}
\right ]_{q^2} \left [
\begin{matrix}
m-2-n\\
l
\end{matrix}
\right ]_{q^2}.
\end{align*}}
\end{corollary}

\begin{proof}
In \eqref{R2eq2}, replace $q$ with $q^2$, set $x=q^{2m-1}$, $y=-1$,
$z=q$ and combine with \eqref{rrgen33}.
\end{proof}

\section{$m$-versions of identities when $m$ is a negative integer}

In \cite{IS03} and \cite{G05}, the authors state some $m$-versions of identities, where
$m$ is a negative integer. We are also able to easily derive a large number of similar identities using our continued fraction approach. We first prove two general transformation, the negative $m$ versions of those in Theorems \ref{t1ef} and \ref{t2ef}.

\begin{theorem}\label{t3ef}
Let $m \geq 1$ be a positive integer.

(i) Let $\phi(x,y,z,q)$ and $e_m(x,y,z,q)$ be as defined in Theorem \ref{t1ef}. Then
\begin{multline}\label{mn1eq}
\phi(x q^{-m},y,z,q)\\=e_{m+1}(xq^{-m},y,z,q)\phi(x,y,z,q)+(zx-y)e_{m}(xq^{-m},y,z,q)\phi(xq,y,z,q).
\end{multline}

(ii) Let $\Phi(x,y,z,q)$ and $g_m(x,y,z,q)$ be as defined in Theorem \ref{t2ef}. Then
\begin{multline}\label{mn2eq}
\Phi(x q^{-m},y,z,q)=\frac{q^{m(m-1)/2}}{x^my^m(-1/xy;q)_m}
\bigg[g_{m+1}(xq^{-m},y,z,q)\Phi(x,y,z,q)\\+\frac{zx-x^2yq}{1+xyq}g_{m}(xq^{-m},y,z,q)\Phi(xq,y,z,q)
\bigg].
\end{multline}
\end{theorem}

\begin{proof}
The first transformation follows immediately upon replacing $x$ with $xq^{-m}$ in \eqref{server}, while the second follows similarly upon making the same substitution in
\eqref{server2}, and then employing the identity (see \cite[I.8, page 351]{GR04})
\[(-xyq^{1-m};q)_m =(-1/xy;q)_m (xy)^m q^{m(m-1)/2}.
\]
\end{proof}

We now give some examples of negative $m$-versions of identities. The negative $m$ version of the Rogers-Ramanujan identities stated by Ismail and Stanton (\textbf{case (5.4e)} in \cite{IS03}) is an
easy consequence of \eqref{mn1eq} (our statement of this result is slightly different).

\begin{corollary}\label{cm1}
For each positive integer $m$,
\begin{equation}\label{rrgenm}
\sum_{n=0}^{\infty}\frac{q^{n^2-m n}}{(q;q)_n}=\frac{ a_m(q)}{(q,q^4;q^5)}+\frac{b_m(q)}{(q^2,q^3;q^5)},
\end{equation}
where
\begin{align*}
a_m(q)&=\sum_{n} q^{n^2-m n}\left [
\begin{matrix}
m-n\\
n
\end{matrix}
\right ],\\
b_m(q)&=\sum_{n} q^{n^2-m n}\left [
\begin{matrix}
m-1-n\\
n
\end{matrix}
\right ].
\end{align*}
\end{corollary}
\begin{proof}
In \eqref{mn1eq}, set $z=1/x$, let $x,y \to 0$, and use \eqref{rridseq}.
\end{proof}

Remark: Polynomial generalizations of negative $m$-versions of identities may also be easily derived, by
replacing $x$ with $x q^{-m}$ in \eqref{polyver} and \eqref{polyver2}, and then specializing $x$, $y$ and $z$ as before, but we do not investigate that here. We next give a negative $m$-version of the G\"{o}llnitz-Gordon identities.

\begin{corollary}\label{cm4}
For $|q|<1$ and integral $m\geq 1$,
\begin{equation}\label{rrgen4m}
\sum_{n=0}^{\infty}\frac{(-q;q^2)_nq^{n^2-2m n}}{(q^2;q^2)_n}
=\frac{a_m(q)}{(q^3,q^4,q^5;q^8)_{\infty}}
+\frac{b_m(q)}{(q,q^4,q^7;q^8)_{\infty}},
\end{equation}
where
\begin{align*}
a_m(q)&=\sum_{n,l} q^{n^2+l^2-2mn}\left [
\begin{matrix}
m-1-l\\
n
\end{matrix}
\right ]_{q^2} \left [
\begin{matrix}
n\\
l
\end{matrix}
\right ]_{q^2},\\
b_m(q)&=\sum_{n,l} q^{n^2+l^2-2mn}\left [
\begin{matrix}
m-l\\
n
\end{matrix}
\right ]_{q^2} \left [
\begin{matrix}
n\\
l
\end{matrix}
\right ]_{q^2}.
\end{align*}
\end{corollary}

\begin{proof}
In \eqref{mn1eq}, replace $q$ with $q^2$, set $x=1/q$, $y=0$,
$z=q$ and use \eqref{ggidseq}.
\end{proof}

We next give a negative $m$-version of the identity in Corollary \ref{c3}.

\begin{corollary}\label{c3m}
For $|q|<1$ and integral $m\geq 1$,
\begin{multline}\label{rrgen3m}
\sum_{n=0}^{\infty}\frac{q^{n^2-2m
n}}{(q^4;q^4)_n}
=\frac{a_m(q)}{(q^2,q^3;q^5)_{\infty}(-q^2;q^2)_{\infty}}
+\frac{b_m(q)}{(q,q^4;q^5)_{\infty}(-q^2;q^2)_{\infty}},
\end{multline}
where
\begin{align*}
a_m(q)&=\sum_{n,j} q^{n^2-2mn}(-1)^j\left [
\begin{matrix}
m-1-j\\
n
\end{matrix}
\right ]_{q^2} \left [
\begin{matrix}
n+j\\
j
\end{matrix}
\right ]_{q^2},\\
b_m(q)&=\sum_{n,j} q^{n^2-2mn}(-1)^j\left [
\begin{matrix}
m-j\\
n
\end{matrix}
\right ]_{q^2} \left [
\begin{matrix}
n+j\\
j
\end{matrix}
\right ]_{q^2}.
\end{align*}
\end{corollary}

\begin{proof}
In \eqref{mn1eq}, replace $q$ with $q^2$, set $x=1/q$, $y=-1$,
$z=0$ and use the
 identities at \eqref{c3idseq}.
\end{proof}

Note that Watson's identity \eqref{wateq1} does not give anything very interesting when applied to negative $m$-versions of identities for which $z=1/x$, since replacing $x$ with
$x q^{-m}$ and then setting $z=1/x$ causes the product $(zxq;q)_{\infty}$ on the right at
\eqref{wateq1} to vanish. However, we do get a new identity when $z \not =1/x$. We give a negative $m$-version of Corollary \ref{c3w} as an example.

\begin{corollary}\label{c3wm}
For $|q|<1$ and integral $m\geq 1$,
\begin{multline}\label{rrgen3wm}
\sum_{n=0}^{\infty}\frac{q^{3n^2-2m
n}(-1)^n}{(-q^{1-2m};q^2)_{n}(q^4;q^4)_n}\\
=\frac{q^{m^2}}{(-q,q^2)_m}\bigg[\frac{a_m(q)}{(q^2,q^3;q^5)_{\infty}(-q;q)_{\infty}}
+\frac{b_m(q)}{(q,q^4;q^5)_{\infty}(-q;q)_{\infty}}\bigg],
\end{multline}
where
\begin{align*}
a_m(q)&=\sum_{n,j} q^{n^2-2mn}(-1)^j\left [
\begin{matrix}
m-1-j\\
n
\end{matrix}
\right ]_{q^2} \left [
\begin{matrix}
n+j\\
j
\end{matrix}
\right ]_{q^2},\\
b_m(q)&=\sum_{n,j} q^{n^2-2mn}(-1)^j\left [
\begin{matrix}
m-j\\
n
\end{matrix}
\right ]_{q^2} \left [
\begin{matrix}
n+j\\
j
\end{matrix}
\right ]_{q^2}.
\end{align*}
\end{corollary}

\begin{proof}
In \eqref{wateq1}, replace $q$ with $q^2$ and $x$ with $xq^{-2m}$,
and then set $x=1/q$, $y=-1$ and   $z = 0$. Combine the resulting
identity with \eqref{rrgen3m}, and now \eqref{rrgen3wm} follows after a little manipulation.
\end{proof}

Heine's transformation \eqref{H1eq} and Ramanujan's transformation \eqref{R2eq} may  be similarly used in conjunction with some existing negative $m$-versions of identities to produce new negative $m$-versions.

\begin{corollary}\label{c3hm}
For $|q|<1$ and integral $m\geq 1$,
\begin{multline}\label{rrgen3hm}
\sum_{n=0}^{\infty}\frac{q^{n^2+
n}}{(-q^{1-2m};q^2)_{n}(q^2;q^2)_n}\\
=\frac{q^{m^2}}{(-q;q^2)_m}\bigg
[\frac{a_m(q)}{(q^2,q^3;q^5)_{\infty}(-q;q^2)_{\infty}}
+\frac{b_m(q)}{(q,q^4;q^5)_{\infty}(-q;q^2)_{\infty}}\bigg ],
\end{multline}
where $a_0(q)=0$, $b_0(q)=1$, and for $m\geq 1$,
\begin{align*}
a_m(q)&=\sum_{n,j} q^{n^2-2mn}(-1)^j\left [
\begin{matrix}
m-1-j\\
n
\end{matrix}
\right ]_{q^2} \left [
\begin{matrix}
n+j\\
j
\end{matrix}
\right ]_{q^2},\\
b_m(q)&=\sum_{n,j} q^{n^2-2mn}(-1)^j\left [
\begin{matrix}
m-j\\
n
\end{matrix}
\right ]_{q^2} \left [
\begin{matrix}
n+j\\
j
\end{matrix}
\right ]_{q^2}.
\end{align*}
\end{corollary}
\begin{proof}
In \eqref{H1eq}, replace $q$ with $q^2$ and $x$ with $x q^{-2m}$, set $x=1/q$ and
$y=-1$, let $z \to 0$ and combine with \eqref{rrgen3m}.
\end{proof}

\begin{corollary}\label{cm4r}
For $|q|<1$ and integral $m\geq 1$,
\begin{multline}\label{rrgen4mr}
\sum_{n=0}^{\infty}\frac{q^{2n^2-2m n}}{(-q^{1-2m};q^2)_n(q^2;q^2)_n}\\
=\frac{q^{m^2}}{(-q;q^2)_m}\bigg
[\frac{a_m(q)}{(q^3,q^4,q^5;q^8)_{\infty}(-q;q^2)_{\infty}}
+\frac{b_m(q)}{(q,q^4,q^7;q^8)_{\infty}(-q;q^2)_{\infty}}\bigg],
\end{multline}
where
\begin{align*}
a_m(q)&=\sum_{n,l} q^{n^2+l^2-2mn}\left [
\begin{matrix}
m-1-l\\
n
\end{matrix}
\right ]_{q^2} \left [
\begin{matrix}
n\\
l
\end{matrix}
\right ]_{q^2},\\
b_m(q)&=\sum_{n,l} q^{n^2+l^2-2mn}\left [
\begin{matrix}
m-l\\
n
\end{matrix}
\right ]_{q^2} \left [
\begin{matrix}
n\\
l
\end{matrix}
\right ]_{q^2}.
\end{align*}
\end{corollary}

\begin{proof}
In \eqref{R2eq}, replace $q$ with $q^2$ and $x$ with $x q^{-2m}$, set $x=1/q$, $y=0$,
$z=q$ and combine with \eqref{rrgen4m}.
\end{proof}

All of the negative $m$-versions so far in this section have followed from either \eqref{mn1eq}, or from \eqref{mn1eq} in conjunction with one of the transformations at \eqref{wateq1}, \eqref{H1eq} or \eqref{R2eq}. We now consider \eqref{mn2eq}. We limit this consideration to a negative $m$-version of \eqref{rrgen11}, although we could have derived negative $m$-versions of \eqref{rrgen22} and \eqref{rrgen33}, and also derived yet further negative $m$-versions by then applying \eqref{wateq1}, \eqref{H1eq} or \eqref{R2eq} to each of those identities.

\begin{corollary}\label{cc1m}
For $|q|<1$ and integral $m\geq 1$,
\begin{multline}\label{rrgen11m}
\sum_{n=0}^{\infty} \frac{q^{n^2-2m n}}{(q^{1-2m};q^2)_{n}(q^2;q^2)_n}=\\
\frac{q^{m^2}(-1)^{m}(-q;q^2)_{\infty}}{(q;q^2)_m(q^2;q^2)_{\infty}}\bigg[a_m(q)(q^4,q^{16},q^{20};q^{20})_{\infty}
+b_m(q)(q^8,q^{12},q^{20};q^{20})_{\infty}\bigg],
\end{multline}
where
\begin{align*}
a_m(q)&=\sum_{n,j} q^{n^2-2mn}(-1)^j \left [
\begin{matrix}
m-1-n+j\\
j
\end{matrix}
\right ]_{q^2} \left [
\begin{matrix}
m-1-j\\
n-j
\end{matrix}
\right ]_{q^2},\\
b_m(q)&=\sum_{n,j} q^{n^2-2mn}(-1)^j \left [
\begin{matrix}
m-n+j\\
j
\end{matrix}
\right ]_{q^2} \left [
\begin{matrix}
m-j\\
n-j
\end{matrix}
\right ]_{q^2}.
\end{align*}
\end{corollary}

\begin{proof}
In \eqref{mn2eq} replace $q$ with $q^2$ and $x$ with $x q^{-2m}$, set $x=1/q$, $y=-1$,
$z=0$ and use the
 identities at \eqref{a7996eq}.
\end{proof}

\section{Concluding Remarks}
The authors in \cite{AKP00} give a polynomial generalization of
\eqref{gisrr}  and give similar identities in \cite{AKPP01}. As we have already noted,  similar generalizations of many of the identities in the present paper could
also have been stated. However, we do not state them explicitly
here - the interested reader may easily derive them, if desired, by specializing
$x$, $y$ and $z$ in \eqref{polyver} or \eqref{polyver2}
in the same way that they were specialized in \eqref{phirecureq1} or\eqref{phirecureq2}
to produce the identities in the various corollaries.

We point out that the methods outlined in the present paper do not give all of the $m$-versions arising from pairs of identities of Rogers-Ramanujan type that have been derived elsewhere (for example, they do not lead to a proof of
$m$-version in Theorem 2.7 in \cite{GP09}). Our methods are also restricted to $m$-versions of \emph{pairs} of Slater-type identities (since the recurrence relations for the partial  numerators and denominators of continued fractions are \emph{three}-term recurrences), so that they will not lead to $m$-versions of \emph{triples} of Slater-type identities, such as have been given in \cite{GP09} and elsewhere.

It may also be the case that there are other transformations for basic hypergeometric series that may be used to derive further $m$-versions, in ways similar to how those of Watson, Heine and Ramanujan were used in the present paper.

 \allowdisplaybreaks{

}

\end{document}